\documentclass[11pt,twoside]{article}

\usepackage[latin1]{inputenc} 
\usepackage[T1]{fontenc} 
\usepackage{amsmath,amsthm}
\usepackage{amssymb}
\usepackage{graphicx}
\usepackage{color}
\usepackage{epic}
\usepackage{color}

\def\det{\mathrm{det}}

\def\Id{\mathrm{Id}}

\def\d{\, d}

\def\E{\mathbb{E}}

\def\R{\mathbb{R}}

\def\L{\mathbb{L}}

\def\e{\mathrm{e}}
\def\dps{\displaystyle}
\def\ph{\varphi}

\def\one {  {1\hspace{-2.3mm}{1}}  }

\newcommand{\abs}[1]{\left\vert#1\right\vert}
\newcommand{\set}[1]{\left\{#1\right\}}

\newcommand{\pare}[1]{\left(#1\right)}

\newcommand{\sys}[2]{\left\{ \begin{array}{#1} \dps #2\end{array}\right.}
\newcommand{\arr}[2]{\left. \begin{array}{#1} \dps #2\end{array}\right.}

\newcommand{\bmat}{\begin{pmatrix}}
\newcommand{\emat}{\end{pmatrix}}
\newcommand{\el}{\\[8pt]}

\theoremstyle{plain}
\newtheorem{Theorem}{Theorem}[section]
\newtheorem{Lemma}[Theorem]{Lemma}
\newtheorem{Proposition}[Theorem]{Proposition}

\newtheorem{Definition}[Theorem]{Definition}
\newtheorem{Algorithm}[Theorem]{Algorithm}

\theoremstyle{definition}
\newtheorem{Remark}[Theorem]{Remark}
\newtheorem{Example}[Theorem]{Example}

\def\node{\mathcal{N}}
\def\group{\mathcal{S}}
\def\Hilb{\mathcal{H}}
\title{On a probabilistic interpretation of shape derivatives of Dirichlet groundstates with application to Fermion nodes.}
\author{Mathias Rousset\footnote{ INRIA Lille - Nord Europe \& Université Lille 1, Villeneuve d'Ascq, France.\newline E-mail: {\tt mathias.rousset@inria.fr}
}
}

\begin{document}
\maketitle
\tableofcontents

\begin{abstract}
This paper considers Schrödinger operators, and presents a probabilistic interpretation of the variation (or shape derivative) of the Dirichlet groundstate energy when the associated domain is perturbed. This interpretation relies on the distribution on the boundary of a stopped random process with Feynman-Kac weights. Practical computations require in addition the explicit approximation of the normal derivative of the groundstate on the boundary. We then propose to use this formulation in the case of the so-called fixed node approximation of fermion groundstates, defined by the bottom eigenelements of the Schrödinger operator of a fermionic system with Dirichlet conditions on the nodes (the set of zeros) of an initially guessed skew-symmetric function. We show that shape derivatives of the fixed node energy vanishes if and only if either (i) the distribution on the nodes of the stopped random process is symmetric; or (ii) the nodes are exactly the zeros of a skew-symmetric eigenfunction of the operator. We propose an approximation of the shape derivative of the fixed node energy that can be computed with a Monte-Carlo algorithm, which can be referred to as Nodal Monte-Carlo (NMC). The latter approximation of the shape derivative also vanishes if and only if either (i) or (ii) holds.
\end{abstract}

\section{Introduction and results}
Throughout this paper, we consider a Schrödinger operator in $\R^d$ of the form:
\begin{equation}
  \label{eq:H}
  H = -\frac{\Delta}{2} + V,
\end{equation}
with a smooth potential $V$ going to infinity at infinity, and acting on real valued functions generically denoted with the letter $\psi$ ('wave functions'). Such functions $\psi$ will be defined up to a real valued multiplicative constant ({\it e.g.} in eigenvalue and/or variational problems). We also consider a general family 
\begin{equation}
  \label{eq:omega}
  \theta \mapsto \Omega_\theta
\end{equation}
of open smooth domains in $\R^d$ depending sufficiently smoothly of a parameter $\theta \in \R^p$. The boundary will be denoted $\partial \Omega_\theta$. Gradients in the space $\R^d$ will be denoted $\nabla$, and gradients with respect to $\theta$, $\nabla_\theta$. The \emph{Dirichlet groundstate} and its \emph{Dirichlet groundstate energy} $(\psi^\ast_\theta,E_\theta^\ast)$,
are then defined as the unique bottom eigenelements of $H$, solution to the variational problem:
\begin{align}
 E_\theta^\ast &\dps :=    \inf\pare{ 
\frac{ \dps \int_{\Omega_\theta}   \psi  H\pare{\psi}  }{ \dps \int_{\Omega_\theta}   \psi^2 }
, \quad \psi \vert_{\partial \Omega_\theta} = 0 }  \nonumber \\
 & \dps =  \frac{ \dps  \int_{\Omega_\theta}   \psi^\ast_\theta  H\pare{\psi^\ast_\theta }  }{\dps \int_{\Omega_\theta}   (\psi^\ast_\theta)^2 } .  \label{eq:groundstate}
\end{align}
Calculus of variations detailed in Section~\ref{sec:shape} then yields the shape derivative of the groundstate energy through the formula:
\begin{equation}\label{eq:vargroundstate}
\nabla_\theta E_\theta^\ast =   - \frac{1}{2}  \dps \int_{\partial \Omega_\theta}  \abs{ \nabla \psi^*_\theta }^2 r_{\theta} \, d\sigma ,
\end{equation}
where in the above $\sigma$ is the usual surface measure induced by the canonical Euclidean structure $\R^d$, and $r_\theta$ is the \emph{shape derivative}, {\it i.e.} the field
$$
r_\theta: \partial \Omega_\theta \to \R^p
$$
such that formally the boundary variation writes down:
$$
  \partial \Omega_{\theta+d \theta} = \set{x+ n(x)  r_{\theta}(x)  \cdot  d \theta \, \vert \,   x \in \partial \Omega_{\theta}} ,
$$
where $n(x)$ is the exterior normal vector at $x\in \partial \Omega_\theta$. If $\set{ \psi_{\theta} }_{\theta \in \R^p}$ is a smooth family of smooth functions such that $\psi_{\theta}(x)=0$ for $x \in \partial \Omega_\theta$, then the shape derivative $r_\theta$ can also be defined through
\begin{equation}
  \label{eq:boundvar}
  r_{\theta}(x) \nabla \psi_{\theta}(x) \cdot  n(x) = - \nabla_\theta \psi_{\theta}(x) \qquad \forall x\in \partial \Omega_\theta.
\end{equation}
Formula \eqref{eq:boundvar} can be proved as follows: from the Dirichlet conditions, one has for any small $h$:
$$
 \arr{rl}{
0 & = \psi^\ast_{\theta+h}\pare{x + h \cdot r_{\theta}(x) n(x) + {\rm O}(\abs{h}^2)} -\psi^\ast_{\theta}(x)  \\
& =  h \cdot r_{\theta}(x) \nabla \psi^\ast_{\theta}(x) \cdot  n(x) + h\cdot \nabla_\theta \psi^\ast_{\theta}(x) + { \rm O }(h).
}
$$
Differentiability in formula \eqref{eq:vargroundstate} is a classical result of abstract analytic perturbation of linear operators (see \cite{Kat76}), but can be proved directly with the variational formulation as in \cite{GarSab01}.

In Section~\ref{sec:proba}, we introduce a standard Wiener process (Brownian motion)
\[
 t \mapsto W_t,
\]
with some given initial distribution in $\Omega_\theta$. The first exit time of the domain $\Omega_{\theta}$ is denoted by
\begin{equation}\label{eq:tau}
 \tau := \inf\pare{t \geq 0 | W_t \in \partial \Omega_\theta}.
\end{equation}
 The long time probability distribution of the latter process with Feynman-Kac weights, and conditioned to remain in the domain $\Omega_{\theta}$ is denoted $d \eta^\ast_{\theta}$
 \begin{equation}
   \label{eq:eta}
   \arr{rl}{
 \int_{\Omega_\theta} \ph \d\eta^\ast_{\theta}  &\dps := \lim_{T \to +\infty} \frac{\dps \E \pare{ \ph(W_T) \one_{ T \leq \tau }  \e^{-\int_0^T V(W_s) ds }  } }{ \dps  \E\pare{\one_{T \leq \tau } \e^{-\int_0^T V(W_s)ds} } } . 
}
 \end{equation}
 It has a probability density function given by the signed groundstate $\psi^\ast_\theta $:
\begin{equation}
   \label{eq:proba_psistar}
   \arr{rl}{
 \int_{\Omega_\theta} \ph \d\eta^\ast_{\theta}  
&\dps = \frac{ \dps \int_{\Omega_\theta} \ph \, \psi^*_\theta  \d x}{ \dps\int_{\Omega_\theta} \psi^*_\theta \d x},
}
 \end{equation}
and the exponential rate of the evolution of the weighted extinction probability yields the groundstate energy:
\begin{equation}
   \label{eq:proba_Estar}
 \lim_{T \to +\infty} - \frac{1}{T} \ln \E\pare{ \one_{ T \leq \tau  }  \e^{-\int_0^T V(W_s)ds} }= E_\theta^\ast.
\end{equation}
The probabilistic interpretations \eqref{eq:eta}-\eqref{eq:proba_psistar}-\eqref{eq:proba_Estar} have some variants (see \eqref{eq:proba_psistar2}-\eqref{eq:proba_Estar2}) where a drift is added to the Wiener process and the range of the potential $V$ in the Feynman-Kac weight $ \e^{-\int_0^T V(W_s) ds }$ is reduced. This leads to Monte-Carlo methods with some \emph{importance sampling} variance reduction which can efficiently compute the couple $(\psi^\ast_\theta,E_\theta^\ast)$. This method has been widely used and studied in many fields. Special care is required to treat the weight $ \e^{-\int_0^T V(W_s) ds }$ when averages are computed. For instance when several random processes are simulated in parallel, some \emph{re-sampling} of the set of processes has to be carried out at regular time intervals, according to the weights associated with each process. We refer the reader to \cite{HamLesRey94,AssCaf00bis,AssCafKhe00} for applications in Quantum Chemistry (Diffusion Monte-Carlo or Pure Diffusion Monte-Carlo methods), to \cite{DouFreGor01,DouDelJas06} for applications in Bayesian statistics (Sequential Monte-Carlo methods), and to \cite{DelMic00,Del04,DelMic03,Rou06} for the associated mathematical analysis.

Then, we consider the weighted distribution of the Wiener process at the hitting time $\tau$, when the process is initially distributed according to $\eta^\ast_{\theta}$ defined in \eqref{eq:eta}. The latter distribution is denoted $\d\mu^\ast_{\theta,\lambda}$ and reads
\begin{equation}\label{eq:mu}
\arr{rl}{
\int_{\partial \Omega_\theta} \ph \d\mu^{\ast}_{\theta,\lambda} := & \E\pare{ \dps \ph(W_\tau) \one_{\tau < + \infty } \e^{-\int_0^\tau \pare{ V(W_s)-\lambda } \d s} \, \vert \, {\rm Law}(W_0) = \eta^\ast_{\theta} } \\
= & \dps \lim_{T \to +\infty}  \frac{\dps  \E\pare{ \ph(W_\tau) \one_{T \leq \tau < + \infty } \e^{-\int_0^\tau \pare{ V(W_s)-\lambda } \d s} } }{ \dps\E\pare{ \dps \one_{T \leq \tau } \e^{-\int_0^T \pare{ V(W_s)-\lambda } \d s}} }  
}
\end{equation}
which verifies
\begin{equation}\label{eq:hitdist}
\arr{rl}{
\int_{\partial \Omega_\theta} \ph \d\mu^{\ast}_{\theta,\lambda} 
= &\dps -\frac{1}{\dps 2(E^\ast_\theta - \lambda   ) \int_{\Omega_\theta} \psi^*_\theta} 
   \int_{\partial \Omega_\theta} \ph \nabla \psi^*_\theta \cdot n \d\sigma .
}
\end{equation}
Formula \eqref{eq:hitdist} holds for any $\lambda < E^*_\theta$ and is the grounding formula of this paper. Note that it can be related to the Dirichlet energy variation $\nabla_\theta E_\theta^\ast $ in \eqref{eq:vargroundstate} by remarking that $\abs{\nabla \psi^*_\theta \cdot n} =\abs{\nabla \psi^*_\theta}$.  Up to our knowledge, the formula \eqref{eq:hitdist} has never been pointed out in the literature, although the sensitivity analysis carried out in \cite{CosGobElk06} yields a similar formula but at finite time (as opposed to large time, which is the case here).

Approximations of $\nabla_\theta E_\theta^\ast$ with Monte-Carlo methods can then be carried out using an approximating
sequence of the Dirichlet groundstate \eqref{eq:groundstate}  $\psi^n_\theta \xrightarrow[]{n \to \infty}{\psi^\ast_\theta}$, and random samples of size $N$ approximating the probabilistic formulations \eqref{eq:eta}-\eqref{eq:mu}, denoted $\eta^{N}_{\theta} \xrightarrow[]{N \to \infty} \eta^\ast_{\theta}$ and  $\mu^{N}_{\theta,\lambda}  \xrightarrow[]{N \to \infty}  \mu^\ast_{\theta,\lambda}$. The following identity can then be used:
\begin{equation}\label{eq:hitdistMC}
\arr{rl}{
\nabla_\theta E_\theta^\ast &\dps =   \frac{(E^\ast_\theta - \lambda   )}{\int_{\Omega_\theta} \psi^\ast_\theta \d \eta^\ast_{\theta}} \int_{\partial \Omega_\theta} r_\theta \nabla \psi^\ast_\theta.n \d \mu^\ast_{\theta,\lambda} \el
&\dps \sim \frac{(E^\ast_\theta - \lambda   )}{\int_{\Omega_\theta} \psi^n_\theta \d \eta^{N}_{\theta}} \int_{\partial \Omega_\theta} r_\theta \nabla \psi^n_\theta.n \d \mu^{N}_{\theta,\lambda} 
}
\end{equation}
The main limitation of computing $\nabla_\theta E_\theta^\ast$ with the Monte-Carlo technique suggested above is the necessity of an analytical approximation $\psi^n_\theta$ of the groundstate $\psi^\ast_\theta$. Especially, the pointwise convergence of the normal derivative $\nabla \psi^n_\theta.n $ may be hard to achieve. However, for practical situations in high dimension, we do not know any alternative point of view. A clear motivating example of a high dimensional problem is the case of Fermionic systems where the so-called \emph{fixed node approximation} is used. 

In Section~\ref{sec:fermion}, we introduce Fermionic groundstates $(\psi_{\rm F}^\ast,E_{\rm F}^\ast)$  associated to a finite symmetry group $\group \subset O(\R^d)$ of the Hamiltonian $H$ in \eqref{eq:H}; where $O(\R^d)$ denotes the group of isometries. $\group $ is simply the permutation goup of identical particles for physical systems. Fermionic groundstates are the solutions to the variational problem:
\begin{equation}\label{eq:fermigroundstate}
\arr{rl}{
 E_{\rm F}^\ast& := \dps   \inf\pare{ \frac{ \dps \int_{\R^d}   \psi  H\pare{\psi} }{ \dps  \int_{\R^d}   \psi^2  } , \quad \forall S \in \group, \, \, \psi\circ S = \det (S) \, \psi }  \\
&\dps =  \frac{ \dps \int_{\Omega_\theta}   \psi^\ast_{\rm F}  H\pare{\psi^\ast_{\rm F} } }{ \dps \int_{\Omega_\theta}   (\psi^\ast_{\rm F})^2 }  .
}
\end{equation}
Any function $\psi$ verifying the symmetry property
$$
\forall S \in \group, \, \, \psi\circ S = \det(S) \psi
$$
will be called \emph{skew-symmetric}, whereas any function $\psi$ verifying
$$
\forall S \in \group, \, \, \psi\circ S =  \psi
$$
will be called \emph{symmetric}. Note that existence of $(\psi_{\rm F}^\ast,E_{\rm F}^\ast)$ follows in our context from the fact that $H$ has a discrete spectrum, and is a classical result of spectral theory for more general potential $V$ (see references in \cite{CanJouLel06}); but uniqueness does not hold in general. In practice, $(\psi_{\rm F}^\ast,E_{\rm F}^\ast)$ is computed using a parametrization of skew-symmetric functions, and a numerical optimization procedure. This is the main problem of computational Quantum Chemistry, and forms a huge scientific field. We refer to \cite{CanLebMad06} for a mathematical introduction with a consequent bibliography. See also the following two typical papers \cite{UmrFil05,TouUmr07} involving wave function optimization using a Monte-Carlo method. Monte Carlo methods in computational Quantum Chemistry are referred to as Quantum Monte Carlo (QMC) methods. For physical systems, the parametrization is given as a finite sum of Slater determinants multiplied by a strictly positive symmetric factor (called the \emph{Jastrow factor}). The result of the latter optimization relies crucially on the quality of the parametrization, and will be called the \emph{trial wave function}, which is a skew-symmetric function denoted $\psi^{\rm I}_{\theta_0}$. We will then consider a family of skew-symmetric functions $\set{ \psi_{\theta}^{\rm I}}_{\theta\in \R^p}$, with an explicit analytical expression, which includes $\psi^{\rm I}_{\theta_0}$. The latter is used to define the \emph{nodal domains}
\[
 \Omega_\theta = \node_\theta^{+} \cup \node_\theta^{-},
\]
where:
\begin{equation}
   \label{eq:nodal}
   \sys{l}{
\node_\theta^{+} = \set{x \in \R^d \, \vert \, \psi_{\theta}^{\rm I}(x) > 0} \el
\node_\theta^{-} = \set{x \in \R^d \, \vert \, \psi_{\theta}^{\rm I}(x) < 0} \el
\partial \node_\theta = \set{x \in \R^d \, \vert \, \psi_{\theta}^{\rm I}(x) = 0}.
}
 \end{equation}
The \emph{fixed node approximation} consists then in computing with a Monte-Carlo method the Dirichlet groundstate $$(\psi^{FN}_\theta,E^{FN}_\theta)=(\psi^{\ast}_\theta,E^{\ast}_\theta)$$ of the variational problem with Dirichlet conditions in $\node_\theta^{+} \cup \node_\theta^{-}$. We refer the reader to \cite{CepCheKal77,CepAld80} for historical papers on the Monte-Carlo computation of fixed node groundstates. The latter variational problem reads explicitely:
\begin{equation}\label{eq:fixenodegroundstate}
\arr{rl}{
 E_{\theta}^{\rm FN}& := \dps   \inf\pare{ \frac{ \dps \int_{\R^d}   \psi  H\pare{\psi} }{ \dps \int_{\R^d}   \psi^2 }, \quad  \psi\vert_{\partial \node_\theta} = 0 }  \\
&\dps = \frac{  \dps  \int_{\R^d}   \psi^{\rm FN}_{\theta}  H\pare{\psi^{\rm FN}_{\theta} } }{ \dps  \int_{\R^d}   (\psi^{\rm FN}_{\theta})^2    } \\
&\dps = \dps   \inf\pare{  \frac{  \dps \int_{\node_\theta^{+}}   \psi  H\pare{\psi} }{ \dps \int_{\node_\theta^{+}}   \psi^2}  , \quad  \psi\vert_{\partial \node_\theta} = 0 }  \\
&\dps = \frac{  \dps  \int_{\node_\theta^{+}}   \psi^{\rm FN}_{\theta}  H\pare{\psi^{\rm FN}_{\theta} } }{ \dps  \int_{\node_\theta^{+}}  (\psi^{\rm FN}_{\theta})^2},
}
\end{equation}
the same definition holding in $\node_\theta^{-}$ by symmetry. From now on, $\node_\theta^{+}$ and $\node_\theta^{-}$ will be called respectively the positive and negative \emph{nodal domains}, $\partial \node_\theta$ the \emph{nodal surface} or \emph{nodes}, and $( E_{\theta}^{FN},\psi_{\theta}^{FN})$ the \emph{fixed node groundstate} elements. Now the key problem is the following: the energy (called Variational Monte Carlo, or in short VMC energy in the QMC literature)
\begin{equation}\label{eq:VMC_energy}
  E^{\rm I}_\theta = \frac{\dps \int_{\node_\theta^{+}}   \psi^{\rm I}_{\theta}  H\pare{\psi^{\rm I}_{\theta} }}{\dps \int_{\node_\theta^{+}} \pare{\psi^{\rm I}_{\theta} }^2}
\end{equation}
of the trial wave function $\psi^{\rm I}_\theta$ can be minimized towards the Fermionic groundstate energy $E_{F}^\ast$ using some optimization procedure, and the fixed node energy $E_{\theta}^{\rm FN} \leq E^{\rm I}_\theta$ associated with $\psi^{\rm I}_\theta$ can be computed using a Monte-Carlo method associated with \eqref{eq:proba_Estar}. However, a parameter $\theta$ optimizing the energy $E^{\rm I}_\theta$ of the trial wave function does not in general, for a given parametrization, optimize the fixed node groundstate energy $E_{\theta}^{\rm FN}$. An open problem is now to develop an algorithm that can directly minimize $E_{\theta}^{\rm FN}$. To precise this idea, let us consider the formulation of the exact Fermionic groundstate $(\psi_{F}^\ast,E_{F}^\ast)$ as a variational problem involving the fixed node groundstate $(\psi^{\rm FN}_\theta,E^{\rm FN}_\theta)$ and the nodal surface $\partial \node_{\theta}$, that is to say:
\begin{equation}\label{eq:fermifromfixednode}
\arr{rl}{
E_{\rm F}^{\ast}& := \dps   \inf\pare{ E_{\theta}^{\rm FN}, \quad  \text{$\psi^{\rm FN}_{\theta}$ solution of \eqref{eq:fixenodegroundstate}}, \quad \text{$\psi_{\rm I}^\theta$ skew-symmetric} } \\
&\dps =  \frac{\dps \int_{\R^d}   \psi^\ast_{\rm F}  H\pare{\psi^\ast_{\rm F} } }{ \dps \int_{\R^d}   (\psi^\ast_{\rm F})^2 }  . 
}
\end{equation}
An approach to solve the variational problem \eqref{eq:fermifromfixednode} consists in the computation of the shape derivative of the fixed node groundstate
\begin{equation}
  \label{eq:FN_deriv}
\nabla_\theta E_{\theta}^{\rm FN}  
\end{equation}
using a Monte-Carlo estimation based on \eqref{eq:hitdistMC}.  In this context, the key formulas \eqref{eq:hitdist}-\eqref{eq:hitdistMC} can be rewritten as follows. $t\mapsto W^+_t$ denotes a Brownian motion in $\node_{\theta}^+$, and $\tau^+$ the hitting time of $\partial \node_{\theta}$. Then the probability measure $\d \eta_{\theta}^{\rm FN}$ on $\node_{\theta}^+$ is defined by:
\begin{equation}
  \label{eq:etanode}
  \int_{\node_{\theta}^+} \ph \d \eta_{\theta}^{\rm FN}  :=  \lim_{T \to +\infty} \frac{\dps \E \pare{ \ph(W_T^+) \one_{T \leq \tau^+ }  \e^{-\int_0^T V(W_s^+) ds }  } }{ \dps  \E\pare{\one_{T \leq \tau^+ } \e^{-\int_0^T V(W_s^+)ds} } },
\end{equation}
and the measure $\d \mu_{\theta, \lambda}^{\rm FN}$ on $\partial \node_{\theta} $ for $\lambda<E_\theta^{\rm FN}$ is defined by:
\begin{equation}\label{eq:munode}
\int_{\partial \node_\theta} \ph \d\mu^{\rm FN}_{\theta, \lambda} := \lim_{T \to +\infty}  \frac{\dps \E\pare{ \dps \ph(W_{\tau^+}^+) \one_{T \leq \tau^+ < + \infty} \e^{-\int_0^{\tau^+} \pare{ V(W^+_s)-\lambda } \d s} } }{ \dps\E\pare{ \dps \one_{ T \leq \tau^+} \e^{-\int_0^T \pare{ V(W_s^+)-\lambda } \d s}} }.
\end{equation}
We will show that \eqref{eq:hitdistMC} becomes:
\begin{align}
  \label{eq:partialEFN}
  \nabla_\theta E_\theta^{\rm FN} & = \frac{(E^{\rm FN}_\theta - \lambda   )}{ \int_{\node_\theta^+} \psi^{\rm FN}_\theta \d \eta_\theta^{\rm FN}}  \int _{\partial \node_\theta}  r_\theta^+ \nabla^{+} \psi^{\rm FN}_\theta    \cdot n_+  \d \mu_{\theta,\lambda}^{\rm FN} \nonumber\\
& = \frac{2 (E^{\rm FN}_\theta - \lambda   )}{ \int_{\node_\theta^+} \psi^{\rm FN}_\theta \d \eta_\theta^{\rm FN}}  \int _{\partial \node_\theta}  r_\theta^+ \nabla^{\rm sy} \psi^{\rm FN}_\theta    \cdot n_+  \d \mu_{\theta,\lambda}^{\rm FN}.
\end{align}
In the above, $r_\theta^+$ is the shape derivative of $\node_{\theta}^+$,  $n_+$ is the associated exterior normal vector, and $\nabla^{\rm sy} \psi^{\rm FN}_\theta    \cdot n_+$ is the symmetrization of the normal groundstate gradient, that is to say:
\begin{equation}
    \label{eq:nsym}
    \nabla^{\rm sy} \psi \cdot n_+ = \frac{1}{2}\pare{\nabla^+ \psi \cdot n_+ -  \nabla^- \psi \cdot n_-},
  \end{equation}
where $\nabla^+ \psi \cdot n_+$ (resp. $\nabla^- \psi \cdot n_-$)  is the exterior normal derivative in $\node_\theta^+$ (resp. $\node_\theta^-$). By construction, $r_\theta^+$ is a skew-symmetric field on $\partial \node_\theta$, and $\nabla^{\rm sy} \psi \cdot n_+$ is symmetric. Our main result concerns then the link between (i) a \emph{symmetry breaking} of the measure $\mu_{\theta,\lambda}^{\rm FN}$, (ii) the fact that $\psi^{\rm FN}_\theta$ is an eigenfunction, and (iii) local exrema of $\theta \to E_\theta^{\rm FN}$. The link between (i), (ii) and (iii) can be stated through the following equivalent assertions:
\begin{enumerate}
\item The measure $\mu_{\theta,\lambda}^{\rm FN}$ on $\partial \node_\theta$ is symmetric ({\it i.e.} invariant by the action of $\group$).
\item When defined on the whole space $\R^d$, the gradient of the fixed node  goundstate $\nabla \psi^{\rm FN}_\theta$ is continuous on $\partial \node_\theta$.
\item The fixed node  goundstate $\psi^{\rm FN}_\theta$ is a skew-symmetric eigenfunction of $H$ on $\R^d$.
\item The fixed node energy variation vanishes $\nabla_\theta E_\theta^{\rm FN}=0$ for any parametrization $\theta \mapsto \partial \node_\theta$ of the nodal surface.
\end{enumerate} 
This yields a probabilistic characterization of the nodes (set of zeros) of skew-symmetric eigenstates of $H$ through a symmetry argument. This is an original result. 
The practical computation of $\nabla_\theta E_\theta^{\rm FN}$ using \eqref{eq:partialEFN} requires a Monte-Carlo estimator of the measures $(\mu^{\rm FN}_{\theta, \lambda},\eta^{\rm FN}_{\theta})$ on the one hand, and an analytical approximation of the fixed node groundstate elements $(\psi^{\rm FN}_{\theta}, \nabla^{\rm sy} \psi^{\rm FN}_{\theta} \cdot n_+)$ on the other hand (see also \eqref{eq:hitdistMC}). A key remark is that the elements $(\psi^{\rm FN}_{\theta}, \nabla^{\rm sy} \psi^{\rm FN}_{\theta} \cdot n_+)$ and any approximation with a skew-symmetric function smooth on $\R^d$, for instance with $(\psi^{\rm I}_{\theta}, \nabla \psi^{\rm I}_{\theta} \cdot n_+)$, share the same symmetry properties. This suggests the following approximation of the energy variation $\nabla_\theta E_\theta^{\rm FN}$ in \eqref{eq:partialEFN}:
\begin{eqnarray} 
  \nabla_\theta E_\theta^{\rm FN} \sim  \widehat{\nabla_\theta E_\theta^{\rm FN}}  &= &\dps \frac{ 2(E^{\rm FN}_\theta - \lambda   )}{ \int_{\node_\theta^+} \psi^{\rm I}_\theta \d \eta_\theta^{\rm FN}}  \int _{\partial \node_\theta} r_\theta^+ \nabla \psi^{\rm I}_\theta    \cdot n_+  \d \mu_{\theta,\lambda}^{\rm FN}   \label{eq:partialEFNapproxdef}  \\
& = & -\dps\frac{ 2(E^{\rm FN}_\theta - \lambda   )}{ \int_{\node_\theta^+} \psi^{\rm I}_\theta \d \eta_\theta^{\rm FN}}  \int _{\partial \node_\theta} \nabla_\theta \psi^{\rm I}_\theta  \d \mu_{\theta,\lambda}^{\rm FN}  \label{eq:partialEFNapprox} \\ 
& = &   \dps\frac{\dps \int_{\node_\theta^+} \pare{H-E_\theta^{\rm FN}}(\nabla_{\theta }\psi^{\rm I}_\theta) \d \eta_\theta^{\rm FN}}{\dps \int_{\node_\theta^+} \psi^{\rm I}_\theta \d \eta_\theta^{\rm FN}} . \label{eq:partialEFNapproxbis}
\end{eqnarray}
In the same way as for the exact expression $\nabla_\theta E_\theta^{\rm FN}$, the latter vanishes ($\widehat{\nabla_\theta E_\theta^{\rm FN}}=0$) if $1$, $2$, or $3$ holds. In return, if $\widehat{\nabla_\theta E_\theta^{\rm FN}}=0$ for any parametrization of the nodes then $1$, $2$ or $3$ holds. Such algorithms may be referred to as Nodal Monte-Carlo. They will require variance reduction techniques exploiting the symmetry structure, in principle such that the variance of $\widehat{\nabla_\theta E_\theta^{\rm FN}}$ scales appropriately to $0$ in the limit where $\mu_{\theta,\lambda}^{\rm FN}$ becomes symmetric (in other words, we seek for an 'asymptotically scaling variance reduction', also called  'zero-variance estimation' in the QMC literature). Ideas are given for future work on this matter.

Let us now position the content of this paper in the context of the QMC literature. First, we recall that efficiently computing in high dimension $d \gg 1$ the Fermionic groundstate \eqref{eq:fermigroundstate} using Monte-Carlo methods is a fundamental problem with many applications; for instance it amounts to solve the eigenvalue problem for excited eigenstates, where classical power methods fail. A general solution is known to be intractable, and is usually referred to as the \emph{sign problem} (see Remark~\ref{rem:signpb}). This explains the necessity of the fixed node approximation. The issue of optimizing the nodes of the trial wave function $\psi^{\rm I}_\theta$ in the fixed node approximation was pointed out in \cite{Cep91}, where a long discussion on the structure of Fermion nodes and appropriate (from this perspective) trial wave functions is provided. Yet state-of-the art numerical methods optimizing the trial wave function is based on either, (i) the minimization of the VMC energy $E^{\rm I}_\theta  $ in \eqref{eq:VMC_energy}, as in \cite{UmrFil05,TouUmr07}; or (ii) the minimization of the variance of the local energy
\begin{equation}\label{eq:Eloc}
E_{L} := V - (\psi^{\rm I}_\theta)^{-1}\frac{\Delta}{2}(\psi^{\rm I}_\theta)
\end{equation}
as in \cite{UmrFil05,HonLiu05}. As a consequence, an efficient method optimizing directly the fixed node energy $E^{\rm FN}_\theta$ in \eqref{eq:fixenodegroundstate}, that is to say the nodes of the trial wave function $\psi^{\rm I}_\theta$, remains an unsolved problem and motivates the content of this paper. However, methods to approximately compute the gradient $\nabla_\theta E_\theta^{\rm FN}$ were already suggested in the QMC literature in the more general context of the calculation of physical properties (or ''forces''). The main goal is to compute the derivative $\nabla_R E^{\ast}_{\rm F, \rm R}$ of the groundstate energy with respect to some parameter $R$ parametrizing the Hamiltonian, typically the potential energy $V\equiv V_R$. A classical example is the following: $R$ is the vector of the nuclei-nuclei distances in molecules. The trial wave function now also depends on $R$: $\psi^{\rm I}_{\theta} \equiv \psi^{\rm I}_{R,\theta}$. Well established methods are available for the exact and variance reduced computation of the variational energy gradient:
\begin{equation}
  \label{eq:EI_R}
  \nabla_{R} E^{\rm I}_{\theta,R} 
\end{equation}
where $ E^{\rm I}_{\theta,R}$ is defined by \eqref{eq:VMC_energy}. For instance, methods using coupling (correlated sampling) to reduce variance were tackled in \cite{FilUmr00}, and a general construction of variance reduced estimators ('zero variance-zero bias' estimators) was proposed in \cite{AssCaf03}; see also some applications in  \cite{TouAssUmr07}. The case of the fixed node energy gradient:
\begin{equation}
  \label{eq:EFN_R}
  \nabla_{R} E^{\rm FN}_{\theta,R} 
\end{equation}
where $ E^{\rm FN}_{\theta,R}$ is defined by \eqref{eq:fixenodegroundstate} is more intricate, and requires some approximation since the derivatives $\nabla_R \psi^{\rm FN}_{\theta,R}$ of the fixed node groundstate remain unknown. Note that computing the gradients $\nabla_{\theta} E^{\rm I}_{\theta}$ or $\nabla_{\theta} E^{\rm FN}_{\theta,R}$ can be seen as a particular case of computing respectively \eqref{eq:EI_R} or \eqref{eq:EFN_R}, since the nodes of the trial wave function $\psi^{\rm I}_{R,\theta}$ depend on $R$ {\it a priori} . An approximate formula has already been proposed to compute \eqref{eq:EFN_R}, see for instance equation $(54)$ in \cite{AssCaf03}, $(10)$ in \cite{CasMelRap03}, $(10)$ in \cite{BadNeeds08}, and $(14)$ in \cite{BadHayNed08}. The terms due to the variation of the nodes (which amounts to evaluate $\nabla_{\theta} E^{\rm FN}_{\theta,R}$) are called 'nodal Pulay terms', and were particularly pointed out in \cite{BadNeeds08,BadHayNed08}. In the above references, the formula used for the calculation is exactly the formula \eqref{eq:partialEFNapproxbis}. However, the interpretation in terms of a stochastic process stopped on the nodes in \eqref{eq:partialEFNapprox}-\eqref{eq:partialEFNapproxdef}, the analysis of the symmetry of the associated distribution on the nodes, and the suggestion of an associated Nodal Monte Carlo method are new results.

\vspace{.5cm}
The following classical textbooks are recommended:
\begin{itemize}
\item about spectral theory of operators: \cite{ReeSim78};
\item about elliptic theory of Partial Differential Equations: \cite{GilTru83};
\item about random processes and Feynman-Kac representations: \cite{KarShr91};
\item about Monte-Carlo methods in Quantum Chemistry (QMC): \cite{HamLesRey94,AssCaf00bis}.
\end{itemize}

\section{Shape derivatives of Dirichlet groundstates}\label{sec:shape}
In this section, some notations and results are recalled concerning Dirichlet groundstates of Schrödinger operators with a smooth potential. Then, formula \eqref{eq:vargroundstate} 
is proven formally, and references are given for rigorous proofs.

Consider the Schrödinger operator \eqref{eq:H} defined on $\R^d$ with a smooth potential $V$ bounded from below. $H$ defines a self-adjoint operator on the Hilbert space $\L^2(\R^d)$. For simplicity, $V$ is assumed to go to infinity at infinity such that $H$ has a compact resolvent, and thus a purely discrete spectrum. Generalization to operators involving continuous spectrum, although of fundamental importance in Quantum Chemistry, is left as technical extensions of the present work. Then a parametrization of smooth domains \eqref{eq:omega} is considered for $\theta \in \R^p$ such that $\partial \Omega_\theta$ has a smooth boundary for any $\theta$. One assumes then that there exists a set of diffeomorphisms smoothly indexed by $\theta$ and such that:
\begin{equation}
\label{eq:smoothdiffeo}
\sys{l}{
 (\theta,x) \mapsto R_\theta(x)  \text{ is smooth in $\R^p \times \R^d $}  \el
 R_0 = \Id  \el
R_\theta = \Id  \text{ outside some compact set} \el
\forall \theta \in \R^p, \quad \Omega_\theta = R_{\theta}(\Omega_0).  
}
\end{equation}
In this setting, the \emph{shape derivative} of $\theta \mapsto \Omega_\theta$ can be defined as the smooth field:
$$
r_{\theta}: \partial \Omega_\theta \to \R^p,
$$
verifying:
\begin{equation}
  \label{eq:shapederiv}
\forall x \in  \partial \Omega_\theta, \quad r_{\theta}(x) = \nabla_\theta R_\theta(  x )  \cdot n(x) ,
\end{equation}
and thus locally for small $h\in \R^p$:
$$
\partial \Omega_{\theta+h} \sim \set{ x + h \cdot r_{\theta}(x) n(x) \vert \, x \in \partial \Omega_{\theta}    }.
$$
Now classical results of spectral theory ensures that the Hamiltonian \eqref{eq:H} considered in $\L^2(\Omega_\theta)$ with Dirichlet boundary condition, is self-adjoint with domain $\mathcal D(H) \subset H^1_{0}(\Omega_\theta)$, where $H^1_{0}(\Omega_\theta)$ is the usual Sobolev space of function with Dirichlet conditions and square integrable first order derivatives. Moreover, $H$ has a unique (up to a multiplicative constant) signed groundstate $\psi^{\ast}_\theta $, solution of the variational problem \eqref{eq:groundstate}, or equivalently solution of the eigenvalue problem:
$$
\sys{l}{
 H(\psi^{\ast}_\theta) =E^{\ast}_\theta \psi^{\ast}_\theta  \el
\psi^{\ast}_\theta \vert_{ \partial \Omega_\theta}=  0  \el
\psi^{\ast}_\theta > 0 , 
}
$$
where $\cdot \vert_{ \partial \Omega_\theta}$ denotes the usual trace operator on the boundary. The regularity of $V$ then ensures that $\psi^{\ast}_\theta $ is smooth  on $\Omega_\theta$ and that $\nabla \psi^{\ast}_\theta \cdot n $ is smooth on $\partial \Omega_\theta $.  
The following derivative formula can now be stated.
\begin{Lemma} Consider domains $\theta \mapsto \Omega_\theta$ verifying \eqref{eq:smoothdiffeo}.
Then the Dirichlet energy $E^{\ast}_\theta$ solution of the variational problem \eqref{eq:groundstate} is differentiable with respect to the parameter $\theta$, and the variation formula \eqref{eq:vargroundstate} holds.
\end{Lemma}
\begin{proof} The formal computation is detailed, and references are given for the rigorous result, which is not much more involved but less instructive from our point of view. Fix $\theta_0 \in \R^p$, and consider the formal chain rule:
  \begin{align*}
    \pare{ \nabla_\theta E^\ast_{\theta} }_{\theta=\theta_0}&=\pare{ \nabla_\theta \frac{\dps \int_{\Omega_\theta}   \psi^\ast_\theta  H\pare{\psi^\ast_\theta } }{\dps\int_{\Omega_\theta}   \pare{\psi^\ast_\theta}^2} }_{\theta = \theta_0} \\
&= \pare{  \nabla_\theta \frac{\dps\int_{\Omega_\theta}   \psi^\ast_{\theta_0}  H\pare{\psi^\ast_{\theta_0} } }{ \dps\int_{\Omega_\theta}   \pare{\psi^\ast_{\theta_0}}^2} }_{\theta = \theta_0}
+ \pare{ \nabla_\theta  \frac{\dps\int_{\Omega_{\theta_0}}   \psi^\ast_\theta  H\pare{\psi^\ast_\theta } }{\dps \int_{\Omega_{\theta_0}}   \pare{\psi^\ast_\theta}^2} }_{\theta = \theta_0} .
  \end{align*}
Since $(\psi^\ast_{\theta},E^\ast_{\theta})$ is an eigenelement, it yields:
$$
\nabla_\theta \frac{\dps\int_{\Omega_\theta}   \psi^\ast_{\theta_0}  H\pare{\psi^\ast_{\theta_0} } }{ \dps\int_{\Omega_\theta}   \pare{\psi^\ast_{\theta_0}}^2} 
= \nabla_\theta E^\ast_{\theta_0} = 0.
$$
Then formal differentiation yields
$$
\arr{rl}{
\nabla_\theta  \frac{\dps\int_{\Omega_{\theta_0}}   \psi^\ast_\theta  H\pare{\psi^\ast_\theta } }{\dps \int_{\Omega_{\theta_0}}   \pare{\psi^\ast_\theta}^2}
 & =    \frac{\dps\int_{\Omega_{\theta_0}}   \nabla_\theta\psi^\ast_\theta (H -E^\ast_{\theta})\pare{\psi^\ast_\theta } }{\dps \int_{\Omega_{\theta_0}}   \pare{\psi^\ast_\theta}^2}
 + \frac{\dps\int_{\Omega_{\theta_0}}   \psi^\ast_\theta  (H-E^\ast_{\theta})\pare{\nabla_\theta\psi^\ast_\theta } }{\dps \int_{\Omega_{\theta_0}}   \pare{\psi^\ast_\theta}^2},
}
$$
and since $(\psi^\ast_{\theta},E^\ast_{\theta})$ is an eigenelement, the first term of the right hand side vanishes so that finally:
$$
\arr{rl}{
\pare{ \nabla_\theta E^\ast_{\theta} }_{\theta={\theta_0}}
 & =    \frac{\dps\int_{\Omega_{{\theta_0}}}   \psi^\ast_{{\theta_0}}  (-\frac{1}{2}\Delta + V-E^\ast_{{\theta_0}})\pare{\nabla_{\theta_0}\psi^\ast_{\theta_0} } }{\dps \int_{\Omega_{{\theta_0}}}   \pare{\psi^\ast_{{\theta_0}}}^2} .
}
$$
Now applying Green's integration by parts two times, and remarking that $(H-E^\ast_{{\theta_0}})(\psi^\ast_{{\theta_0}})=0$, we get
$$
\pare{ \nabla_\theta E^\ast_{\theta} }_{\theta=\theta_0}
= \frac{1}{2}  \frac{\dps\int_{\partial \Omega_{{\theta_0}}}   \nabla \psi^\ast_{\theta_0} \cdot n  \, \nabla_{\theta_0} \psi^\ast_{\theta_0}\, d \sigma }{\dps \int_{\Omega_{{\theta_0}}}   \pare{\psi^\ast_{\theta_0}}^2}.
 $$
Then \eqref{eq:boundvar} applied to $\psi^\ast_{{\theta_0}}$ yields the result. The rigorous proof can be made using a change a variable with the diffeomophism $R_\theta$, and then exploiting the variational formulation (see e.g. Theorem~2 in \cite{CosGobElk06}).
\end{proof}

\section{Probabilistic interpretations}\label{sec:proba}
In this Section, the probabilistic formulations \eqref{eq:proba_psistar}-\eqref{eq:proba_Estar}-\eqref{eq:hitdist} are proven and detailed. Associated Monte-Carlo methods are recalled with some references.

Consider notations and assumptions of Section~\ref{sec:shape}.  Let $t \mapsto W_t$ be a standard Wiener process with exit time $\tau$ from $\Omega_{\theta}$ defined in \eqref{eq:tau}. The classical probabilistic interpretation of the eigenelements $(\psi^\ast_{\theta},E^\ast_{\theta})$ is recalled in the following lemma:
\begin{Lemma}
  \label{lem:probaground}
 Assume $W_0$ is distributed according to $\frac{\psi_{\rm init}(x) dx}{\int\psi_{\rm init}(x) dx }$ where $\psi_{\rm init} \in L^2(\Omega_{\theta})$, and $\psi_{\rm init}$ is non vanishing in each connected component of $\Omega_\theta$. The groundstate elements can be expressed through the long time behavior of the process with Feynman-Kac weights and conditioned by large exit times. This yields the formulas \eqref{eq:eta}-\eqref{eq:proba_psistar}-\eqref{eq:proba_Estar}.
\end{Lemma}
\begin{proof}
The result follows from the classical representation of parabolic equations through the Feynman-Kac formula (see \cite{KarShr91}). Let us recall the different steps of the argument. First, consider $\ph \in C_c^{\infty}(\Omega_\theta)$ a smooth solution $(t,x) \mapsto u_t(x) \in  C^{\infty}(\R^+\times\Omega_\theta)$ of the parabolic problem:
$$
\sys{l}{
\partial_t u_t =-  H(u_t)  \el
u_t \vert_{\partial \Omega_{\theta}} = 0  \el
u_0 = \ph 
}
$$
Then using Itô calculus, it yields for any $\phi \in C^{\infty}(\R^+ \times \Omega_{\theta})$:
\begin{equation}
  \label{eq:Ito}
\arr{l}{
  \dps d \pare{ \phi_{T-t}(W_t)\e^{-\int_{0}^{t}V(W_s) ds} } = \\
\dps \qquad\pare{ \frac{\Delta}{2} -V - \partial_t} \pare{ \phi_{T-t} } (W_t)\e^{-\int_{0}^{t}V(W_s) \d s} \d t +\e^{-\int_{0}^{t}V(W_s) \d s} \nabla\phi_{T-t}(W_t) \cdot dW_t,
}
\end{equation}
so that applying the latter computation to $\phi=u$ between time $0$ and the stopping time $T\wedge \tau = \inf(T,\tau)$ yields:
$$
\arr{rl}{
\frac{\dps \int_{\Omega_{\theta}} u_T(x) \psi_{\rm init}(x) dx }{\dps \int_{\Omega_{\theta}} \psi_{\rm init}(x) dx } &= \E\pare{ u_{T-T\wedge \tau}(W_{T \wedge \tau}) \e^{-\int_{0}^{T \wedge \tau}V(W_s) ds} } \\
&= \E\pare{ \ph(W_{T}) \one_{\tau \geq T} \e^{-\int_{0}^{T }V(W_s) ds} } .
}
$$
Now denoting by $(\psi^{\ast,n}_{\theta},E^{\ast,n}_{\theta})_{n \geq 0}$ the full spectrum of $H$ normalized in $\L^2(\Omega_{\theta})$, the above Feynman-Kac representation reads:
\begin{equation}
  \label{eq:specsol}
  \E\pare{ \ph(W_{T}) \one_{\tau \geq T} \e^{-\int_{0}^{T }V(W_s) ds} } = \sum_{n \geq 0} \e^{-E^{\ast,n}_{\theta} T }\frac{\dps \int_{\Omega_{\theta}} \psi^{\ast,n}_{\theta} \psi_{\rm init} }{\dps \int_{\Omega_{\theta}} \psi_{\rm init} } \,\, \int_{\Omega_{\theta}} \psi^{\ast,n}_{\theta} \ph .
\end{equation}
Now using the fact that the groundstate has a spectral gap ($E^{\ast,1}_{\theta} >E^{\ast,0}_{\theta}=E^{\ast}_{\theta}$), the dominant term in \eqref{eq:specsol} when $T \to + \infty$ enables to verify that the groundstate has a sign $\psi^{\ast,0}_{\theta} >0$. Finally, taking the limit $\ph \to \one$ by dominated convergence in the formula above, and then the leading term in the limit $T\to+\infty$ yields \eqref{eq:proba_Estar}. \eqref{eq:proba_psistar} follows.
\end{proof}

In practice however, a diffusion solution to a stochastic differential equation with a repulsive drift at the boundary $\partial \Omega$ is used:
\begin{equation}
  \label{eq:drifteds}
  \sys{l}{
d X_t =(\psi^{\rm I})^{-1}\nabla \psi^{\rm I}(X_t) dt + dW_t  \el
\psi^{\rm I} \vert_{\partial \Omega_{\theta}}  =0,
}
\end{equation}
where $\psi^{\rm I} >0$ is a smooth function strictly positive in $\Omega_{\theta}$ and vanishing on $\partial \Omega_{\theta}$. In \cite{CanJouLel06}, sufficient conditions on the behavior of $\psi^{\rm I}$ near $\partial \Omega_{\theta}$ and at infinity are given for \eqref{eq:drifteds} to be well-posed, and to verify the following variant of \eqref{eq:proba_psistar}-\eqref{eq:proba_Estar}:
\begin{equation}\label{eq:proba_psistar2}
 \lim_{T \to +\infty} \frac{\dps \E \pare{ \ph(X_T)  \e^{-\int_0^T E_{L}(X_s) ds }  } }{ \dps  \E\pare{ \e^{-\int_0^T E_{L}(X_s)ds} } } 
=  \frac{ \dps \int_{\Omega_\theta} \ph \,  \psi^*_\theta \psi^{\rm I} }{ \dps\int_{\Omega_\theta} \psi^*_\theta \psi^{\rm I} },
 \end{equation}
\begin{equation}\label{eq:proba_Estar2}
 \lim_{T \to +\infty} - \frac{1}{T} \ln \E\pare{ \e^{-\int_0^T E_{L}(X_s)ds} }= E_\theta^\ast,
\end{equation}
where the so-called local energy is defined by \eqref{eq:Eloc}.
The proof of the latter probabilistic interpretation \eqref{eq:proba_psistar2}-\eqref{eq:proba_Estar2} is similar to the proof of Lemma~\ref{lem:probaground}, and is based on the mapping of the Hamiltonian $H$ to a weighted $\L^2$ space through:
$$
H_I = (\psi^{\rm I})^{-1} H (\psi^{\rm I} \cdot ) = -\frac{\Delta}{2}  - (\psi^{\rm I})^{-1} \nabla \psi^{\rm I} \cdot \nabla + E_{L}.
$$
Details and assumptions can be found in \cite{CanJouLel06}.

Next the probabilistic interpretation of the shape derivative given by formula \eqref{eq:hitdist} is proven.
\begin{Proposition}
  \label{pro:hitdist}Assume $W_0$ is distributed according to $\frac{\psi_{\rm init}(x) dx}{\int\psi_{\rm init}(x) dx }$ where $\psi_{\rm init} \in L^2(\Omega_{\theta})$, and $\psi_{\rm init}$ is non vanishing in each connected component of $\Omega_\theta$.
  Assume the boundary $\partial \Omega_{\theta}$ is smooth and uniformly Lipschitz, and consider the measure defined for any $\lambda < E^\ast_{\theta}$ by \eqref{eq:mu}. Then \eqref{eq:hitdist} holds.
\end{Proposition}
\begin{proof}
{\it Step $1$.} Let $\ph\in C^{\infty}(\overline{\Omega}_\theta)$. We claim that the parabolic differential equation with inhomogenous Dirichlet conditions
\begin{equation}
  \label{eq:parainhomo}
  \sys{l}{
    \partial_t h_t(\ph) = -\pare{H-\lambda}(h_t(\ph)) \el
    h_t(\ph) \vert _{\partial \Omega_{\theta}} = \ph \vert _{\partial \Omega_{\theta}}  \el
    h_0(\ph) = \ph 
  }
\end{equation}
has a unique smooth solution for $\lambda < E^\ast$ converging exponentially fast towards $h_{\infty}(\ph)$ unique smooth solution of the elliptic inhomogenous Dirichlet problem:
\begin{equation}
  \label{eq:ellinhomo}
  \sys{l}{
 \pare{H-\lambda}(h_{\infty}(\ph)) = 0 \el
h_{\infty}(\ph) \vert _{\partial \Omega_{\theta}} = \ph. 
}
\end{equation}
This is a classical consequence of spectral theory, but we recall briefly the basic arguments. Existence of a smooth solution of \eqref{eq:ellinhomo} follows from the fact that $(H-\lambda)^{-1}$ can be extended to a bounded operator of $\L^2(\Omega_{\theta})$ so that:
$$
h_{\infty}(\ph) = (H-\lambda)^{-1}\pare{H-\lambda}(\ph) - \ph,
$$
which indeed is solution of \eqref{eq:ellinhomo}. Note that in the above $  (H-\lambda)^{-1} $ implicitly refer to the operator with homogenous Dirichlet boundary condition, so that $ (H-\lambda)^{-1} \circ ( H-\lambda ) \neq \Id$ when operating on test function with inhomogenous boundary conditions. Then the homogenous solution of \eqref{eq:parainhomo} has been solved using spectral decomposition in \eqref{eq:specsol}, proving that such homogenous solution is unique and vanishes exponentially fast when $\lambda < E^\ast_{\theta}$. The latter analysis of the homogenous case proves uniqueness in \eqref{eq:parainhomo}-\eqref{eq:ellinhomo}, as well as exponential long time convergence of the time dependent equation \eqref{eq:parainhomo}.

{\it Step $2$.} We claim that for any $\ph\in C^{\infty}(\overline{\Omega}_\theta)$,
\begin{equation}
  \label{eq:hitdistinter}
 \frac{\dps \int_{\Omega_\theta}  h_{\infty}(\ph) \psi_{\rm init} }{\dps \int_{\Omega_\theta}  \psi_{\rm init}} := \E\pare{ \dps \ph(W_{\tau}) \one_{\tau < +\infty} \e^{-\int_0^{\tau} \pare{ V(W_s)-\lambda } \d s}  }
\end{equation}
where $h_{\infty}(\ph)$ is the solution to the elliptic partial differential equation with inhomogenous Dirichlet conditions \eqref{eq:ellinhomo}.
Indeed, using \eqref{eq:Ito} with test function $h_{T-t}(\ph)$, and stopped at time $T\wedge \tau = \inf(T,\tau)$ yields:
$$
\frac{\dps \int_{\Omega_\theta}  h_{T}(\ph) \psi_{\rm init} }{\dps \int_{\Omega_\theta}  \psi_{\rm init}}  := \E\pare{ \dps h_{T-\tau \wedge T}(\ph)(W_{\tau \wedge T }) \e^{-\int_0^{\tau \wedge T} \pare{ V(W_s)-\lambda } \d s} }.
$$
Now the event $\one_{\tau = + \infty}$ has null probability if $t \mapsto W_t$ is recurrent, and if the latter is transcient then using $\dps \lim_{\infty}V = + \infty$:
$$
\lim_{T \to +\infty} \E\pare{ \dps \one_{\tau = +\infty} \e^{-\int_0^{T \wedge \tau} \pare{V(W_s)-\lambda } \d s} } =0;
$$
so that taking the limit $T \to + \infty$ leads to \eqref{eq:hitdistinter}.

{\it Step $3$.} We claim that 
$$
\int_{\Omega_{\theta}} h_{\infty}(\ph) \psi^\ast_{\theta} = -\frac{1}{2(E^\ast_{\theta} - \lambda   )} 
   \int_{\partial \Omega_{\theta}} \ph \nabla \psi^\ast_{\theta}\cdot n \d\sigma.
$$
 Since $\psi^\ast_{\theta}$ is the groundstate:
$$
\psi^\ast_{\theta} = \frac{1}{E^\ast_{\theta} - \lambda}\pare{H-\lambda}(\psi^\ast_{\theta} ),
$$
so that integration by parts yields:
\begin{eqnarray*}
  \int_{\Omega_{\theta}} h_{\infty}(\ph) \psi^\ast_{\theta} &=& -\frac{1}{E^\ast_{\theta} - \lambda} \int_{\partial \Omega_{\theta}} \frac{1}{2}  \ph \nabla \psi^\ast_{\theta} \cdot n \d \sigma \\
&&+\frac{1}{E^\ast_{\theta} - \lambda} \int_{\Omega_{\theta}} \frac{1}{2} \nabla \psi^\ast_{\theta} \cdot \nabla h_{\infty}(\ph)  + (V-\lambda) h_{\infty}(\ph) \psi^\ast_{\theta}  \\
&=& -\frac{1}{E^\ast_{\theta} - \lambda} \int_{\partial \Omega_{\theta}} \frac{1}{2}  \ph \nabla \psi^\ast_{\theta} \cdot n \d \sigma + 0.
\end{eqnarray*}

{\it Step $4$.} 
Now using the Markov property of Brownian motion yields:
$$
\arr{l}{
\E\pare{ \dps \ph(W_\tau) \one_{T \leq \tau < +\infty } \e^{-\int_0^\tau \pare{ V(W_s)-\lambda } \d s} } = \\
\qquad \E\pare{ \E\pare{\dps \ph(\widetilde{W}_{\widetilde{\tau}}) \one_{{\widetilde{\tau}} < +\infty } \e^{-\int_0^{\widetilde{\tau}} \pare{ V(\widetilde{W}_s)-\lambda } \d s} \vert \widetilde{W}_0 = W_T } \one_{T \leq \tau } \e^{-\int_0^T \pare{ V(W_s)-\lambda } \d s} },
}
$$
where $((\widetilde{W}_t)_{t \geq 0},\widetilde{\tau})$ is an independent copy of $(W_t)_{t \geq 0},\tau)$. Together with {\it Step $2$} and  {\it Step $3$} with the probabilistic interpretation of the groundstate $\psi^\ast_\theta$ in  \eqref{eq:mu}-\eqref{eq:proba_psistar}, it completes the proof.
\end{proof}

\section{Fermion groundstates}\label{sec:fermion}
Consider the Schrödinger operator \eqref{eq:H}. Fermionic groundstates $(\psi_{F}^\ast,E_{F}^\ast)$ associated to \eqref{eq:H} are defined with respect to a finite symmetry group $\group \subset O(\R^d)$ of the potential $V$, where $O(\R^d)$ denotes the group of isometries of $\R^d$. A symmetry group of $V$ is defined by the property
$$
\forall S \in \group \quad V \circ S = V.
$$
Fermion systems appears in the special case when $\group$ contains symmetries with odd parity:
$$
\det(S) = -1 ,
$$
so that $H$ can be defined (since the Laplacian operator commutes with isometries) to operate on the Hilbert space of skew-symmetric function:
$$
\Hilb = \set{\psi \in \L^{2}(\R^d) \, \vert \, \forall S \in \group, \, \, \psi\circ S = \det(S) \psi  }.
$$ 
''Fermion'' groundstates are then the solutions to the variational problem \eqref{eq:fermigroundstate}.
\begin{Example}[Fermions]\label{ex:fermions1}
In the case of $N$ physical quantum particles of Fermionic type in dimension $3$, one has $\R^d = \R^{3N}$, and a potential of the form:
$$
V(x) := \sum_{i=1}^{N} V_{\rm ext}(x_i) + \sum_{1\leq i < j \leq N}  V_{\rm int}(x_i-x_j),
$$
where $x=(x_1,\ldots,x_N)$ are the $3$ dimensional coordinates of the particles, $V_{\rm ext}$ is an exterior smooth potential that goes to infinity at infinity, and $V_{\rm int}$ a smooth interaction potential that vanishes at infinity.
Then the discrete symmetry group is the permutation group $\group_N$ of physical particles. See  \cite{CanJouLel06} for the mathematical analyis of Quantum Monte-Carlo (QMC) methods in this context.
\end{Example}
Since in this paper, we restrict to smooth operators with compact resolvent, the existence of $(\psi_{F}^\ast,E_{F}^\ast)$ follows directly from the fact that the spectrum is discrete, and smoothness from the fact $\psi_{F}^\ast$ satisfies an eigenvalue problem of a smooth elliptic operator (see \cite{GilTru83}).
\begin{Remark}[The sign problem]\label{rem:signpb}
Computing directly $(\psi_{F}^\ast,E_{F}^\ast)$ using Monte-Carlo methods is an untractable problem known as the \emph{sign problem}. The latter can be summarized as follows. Since $H$ leaves invariant $\Hilb$, eigenfunctions of $H$ in $\R^d$ are either symmetric or skew-symmetric functions. Thus computing \eqref{eq:specsol} with:
\begin{itemize}
\item a skew-symmetric test function $\psi^{\rm I}$,
\item a non-symmetric positive initial condition $\psi_{\rm init} > 0$,
\end{itemize}
  yields:
\begin{equation*}
    \E\pare{ \psi^{\rm I}(W_{T})  \e^{-\int_{0}^{T }V(W_s) ds} } = 
  \sum_{\psi^{\ast,n} \in \Hilb} \e^{-E^{\ast,n} T } \frac{\dps \int_{\R^d} \psi^{\ast,n} \psi_{\rm init} }{\dps  \int_{\R^d} \psi_{\rm init}} \, \int_{\R^d} \psi^{\ast,n} \psi^{\rm I} ,
\end{equation*}
and in principle if $\ph$ is symmetric:
\begin{equation*}
\lim_{T\to + \infty}\frac{
  \E\pare{ \ph(W_{T}) \psi^{\rm I}(W_{T})  \e^{-\int_{0}^{T }V(W_s) ds} } }{ \E\pare{ \psi^{\rm I}(W_{T})  \e^{-\int_{0}^{T }V(W_s) ds} } }= \frac{\dps \int_{\R^d} \psi^{\ast}_{\rm F} \psi^{\rm I} \ph}{\dps \int_{\R^d} \psi^{\ast}_{\rm F} \psi^{\rm I}} .
\end{equation*}
Unfortunately, this computation relies crucially on the rate of vanishing of the normalisation which is due to the non-symmetry of the initial condition $\psi_{\rm init}>0 $.  Since stochastic processes used to compute such averages quickly forget their initial condition and have a symmetric dynamics in the full state space $\R^d$, one is compelled to compute ratios of vanishing averages (of the type $\frac{0}{0}$), with Monte-Carlo estimators having a non-vanishing statistical variance. This leads to infinite variance when $T$ is large. This forms the sign problem. This problem appears more generally when trying to solve higher eigenvalue problems with Monte-Carlo methods. Although, there is probably no general solution, solving the \emph{sign problem} for particular situations in high dimension would be considered as a major breakthrough.
\end{Remark}
In practice, $(\psi_{F}^\ast,E_{F}^\ast)$ can be approximated using a hybrid methodology in two steps. First, a \emph{trial wave function} is obtained using an analytical parametrization of $\Hilb$, usually of the form:
\begin{equation}
  \label{eq:trial}
  \psi_{\alpha,\theta}^I := J_{\alpha} \psi_{\theta}^{\rm skew}.
\end{equation}
In the above, $J_{\alpha} > 0$ is a strictly positive symmetric part called the \emph{Jastrow factor} and parametrized by $\alpha \in \R^{n}$; and $\psi_{\theta}^{\rm skew}$ is the skew-symmetric part parametrized by $\theta \in \R^{p}$. The most convenient numerical optimization method is then used to solve \eqref{eq:fermigroundstate} in the space of parameters formed by $(\alpha,\theta)$, and the solution of the optimization procedure is denoted with the parameters $(\alpha_0,\theta_0)$. This yields the set of functions $\set{\psi^{\rm I}_\theta}_{\theta \in \R^p}$ defining the nodal domains through
 $$\psi^{\rm I}_\theta :=  J_{\alpha_0} \psi_{\theta}^{\rm skew} . $$
\begin{Example}[Fermions]\label{ex:fermions2}
In the case of $N$ physical particles of Fermionic type, $\psi_{\theta}^{\rm skew}$ is
 built using a sum of Slater determinant, that is to say a sum of functions of the form:
$$
\det \pare{\phi_j(x_i)}_{i,j=1\ldots N},
$$
where $\pare{\phi_j}_{j=1\ldots N}$ are $N$ smooth functions of $\R^3$.
\end{Example}
The link with Section~\ref{sec:shape} and~\ref{sec:proba} is made by posing:
$$
\Omega_{\theta} := \node_\theta^+ \cup \node_\theta^-.
$$
The \emph{Fixed Node Approximation} consists in computing with a Monte-Carlo method the solution $(\psi^{\rm FN}_\theta,E^{\rm FN}_\theta)=(\psi^{\ast}_\theta,E^{\ast}_\theta)$ of the variational problem with Dirichlet conditions \eqref{eq:fixenodegroundstate}.  Such a computation is made using the probabilistic interpretations \eqref{eq:proba_psistar}-\eqref{eq:proba_Estar}, or more usually in practice using the variant \eqref{eq:proba_psistar2}-\eqref{eq:proba_Estar2}. This method is known under the DMC acronym (Diffusion Monte-Carlo) in Computational Chemistry, and has been widely studied, see for instance \cite{HamLesRey94,AssCafKhe00}.

As explained in the introduction, the key problem is that the nodal surface $\partial \node_{\theta} =(\psi_{\theta}^{\rm I})^{-1}(0)$ may be different from $(\psi_{F}^\ast)^{-1}(0)$. The open question is thus now to carrry out numerical methods associated to the variational problem \eqref{eq:fermifromfixednode}. In this context, a direct minimization of  $E_{\theta}^{\rm FN}$ in \eqref{eq:fermifromfixednode} requires the computation of:
$$
\nabla_\theta E_{\theta}^{\rm FN} 
$$
using formula \eqref{eq:vargroundstate} and the probabilistic interpretation \eqref{eq:hitdist}.

\section{Fixed node and symmetry breaking}\label{sec:fixednode}
In this section, we consider the context of Section~\ref{sec:fermion}, and we assume that the nodal surface defined in \eqref{eq:nodal} is a smooth manifold, a sufficient condition being:
\begin{equation}
  \label{eq:smoothnodes}
  \forall x \in \partial \node_\theta, \quad \nabla \psi_{\theta}^{\rm I}(x) \neq 0.
\end{equation}
Moereover, we assume that the assumptions of Section~\ref{sec:shape} on the mapping $\theta \mapsto \node_\theta$ apply, namely that there is a diffeomorphic mapping $\theta \mapsto R_\theta$ associated to the domain $\node_\theta $ such that \eqref{eq:smoothdiffeo} holds. 

A fundamental remark concerns the symmetry of the normal derivative of the fixed node groundstate $\nabla^+ \psi^{\rm FN}_{\theta} .n_+$, (where $\nabla^+$ and $n_+$ refer to the domain $\node_{\theta}^+$, and will be defined below), or equivalently the measure $\mu^\ast_{\theta,\lambda}$ in \eqref{eq:hitdist} defined on $\partial \node_{\theta}$ for stochastic processes evolving in $\node_\theta^+$. The latter indeed presents a \emph{symmetry breaking}, in the sense that they are \emph{only} invariant by the action of the symmetry sub-group 
$$\group^+ := \set{ S \in  \group \, \vert \, \det(S) = 1 } 
$$ on $\partial \node_{\theta}$. Before going further, we will precise notations in appropriate definitions and lemmas.
\begin{Lemma}
  $\group$ is a symmetry group of the nodal surface $\partial \node_{\theta}$ in the sense that any space transformation $S \in \group$ verify:
$$
S(\partial \node_{\theta}) = \partial \node_{\theta}.
$$
\end{Lemma}
\begin{proof}
  By skew-symmetry, it yields for any $x\in \R^d$
$$
\psi^{\rm I}_{\theta}(S(x)) = - \psi^{\rm I}_{\theta}(x)
$$
so that $\psi^{\rm I}_{\theta}(x) =0$ is equivalent to $\psi^{\rm I}_{\theta}(S(x))=0$.
\end{proof}
So consider now $\psi$ a skew-symmetric function in $\R^d$ with Dirichlet boundary conditions on the nodes
$
\psi \vert _{\partial \node_\theta} = 0,
$
and such that the restrictions 
$
\psi \vert_{\node_\theta ^+} : \overline{\node}_\theta ^+ \to \R,
$
on the one hand and 
$
\psi\vert_{\node_\theta ^-} : \overline{\node}_\theta ^- \to \R,
$
on the other hand are smooth on the associated closed domains. Remark that using skew-symmetry, for any $S \in \group$ verifying $\det \, S =-1$, $\psi\vert_{\node_\theta ^-} $ is the image of $\psi \vert_{\node_\theta ^+} $ through the transformation:
$$
\psi\vert_{\node_\theta ^-} = - \psi\vert_{\node_\theta ^+} \circ S .
$$
Two traces of $\nabla \psi$ on $\partial \node_{\theta}$ can then be defined depending if the reference domain is  $\node_\theta ^+$ or $\node_\theta ^-$. The definition of associated symmetric and skew-symmetric traces then follows.
\begin{Definition}
  Let $\psi$ a skew-symmetric function in $\R^d$ with smooth restrictions and Dirichlet boundary conditions on $\overline{\node}_\theta ^+$ and $\overline{\node_\theta} ^-$. $n_+$ denotes the exterior normal of $\partial \node_\theta ^+$, $n_-$ denotes the exterior normal of $\partial \node_\theta ^-$, so that:
  \begin{equation}
    \label{eq:n+-}
    n_+ = - n_- .
  \end{equation}
  $\nabla^+ \psi \cdot n_+$ denotes the exterior normal derivative of $\psi$ in $\node_\theta ^+$, and $\nabla^- \psi \cdot n_-$ the exterior normal derivative of $\psi$ in $\node_\theta ^-$. Then the skew-symmetrization of the normal derivative is defined by
  \begin{equation}
    \label{eq:nskew}
    \nabla^{\rm sk} \psi \cdot n = \frac{1}{2}\pare{\nabla^+ \psi \cdot n_+ +  \nabla^- \psi \cdot n_-},
  \end{equation}
and the plain symmetrization with respect to $n_+$ is defined by \eqref{eq:nsym}.

\end{Definition}
Note that $\psi \in C^{1}(\R^d)$ if and only if $\nabla^{\rm sk} \psi \cdot n  = 0$ on $\partial \node_\theta$, so that $\nabla^{\rm sk} \psi \cdot n$ can be seen as the gradient discontinuity of $\psi$ at $\partial \node_\theta$. One can now precise the idea of symmetry breaking on $\partial \node_\theta$.
\begin{Lemma}\label{lem:symbreak}
  Let $\psi$ a skew-symmetric function in $\R^d$ with smooth restrictions on $\overline{\node}_\theta ^+$ and $\overline{\node}_\theta ^-$, and Dirichlet boundary conditions on $\partial \node_\theta$. Then $\nabla^{\rm sk} \psi \cdot n $ is skew-symmetric, and $\nabla^{\rm sy} \psi \cdot n_+$ is symmetric, under the action of $\group$ on $\partial \node_\theta$. If $\ph \in C^{\infty}_c(\R^d)$, one has the integration by parts formula:
  \begin{equation}
    \label{eq:nodalipp}
  \frac{1}{2} \int_{\R^d} \nabla \ph \cdot \nabla \psi      = \int_{\partial \node _\theta} \ph  \nabla^{\rm sk} \psi \cdot n \d \sigma - \int_{\R^d} \ph \frac{\Delta}{2}(\psi).
  \end{equation}
Moreover, $\nabla^{\rm sk} \psi \cdot n =0$ on $\partial \node _\theta$ if and only if $\nabla^+ \psi \cdot n_+$ or $\nabla^- \psi \cdot n_-$ are symmetric. In the opposite case, $\nabla^+ \psi \cdot n_+$ and $\nabla^- \psi \cdot n_-$  are invariant only under the action of the special sub-group $\group^+ $ and a ''\emph{symmetry breaking}'' occurs.
\end{Lemma}
\begin{proof}
By skew-symmetry of $\psi$, one has for any $S \in \group$ in $\node_\theta^+ \cup \node_\theta^-$:
$$
S \, \nabla \psi \circ S  = \det(S) \, \nabla \psi ,
$$
and then on $\partial \node_\theta$, 
$$
\arr{l}{
 S \, n_+ \circ S = \det(S)\,  n_+  \el
 S \, n_- \circ S = \det(S)\,   n_- ,
}
$$
so that since $S^TS = \Id$, it yields on $\partial \node_\theta$ for any $S \in \group$ :
\begin{equation}
  \label{eq:symprops}
  \sys{ll}{
    \nabla^+ \psi \cdot n_+ \circ S = \nabla^+ \psi \cdot n_+\quad &\text{if $\det(S)=1$}  \el
    \nabla^- \psi \cdot n_- \circ S = \nabla^- \psi \cdot n_-\quad &\text{if $\det(S)=1$}  \el
    \nabla^+ \psi \cdot n_+ \circ S = -\nabla^- \psi \cdot n_-\quad &\text{if  $\det(S)=-1$}.  
  }
\end{equation}

This yields the symmetry properties of $\nabla \psi^{\rm sk} \cdot n$ and $\nabla \psi^{\rm sy} \cdot n_+$ with respect to $\group$, and of $\nabla^+ \psi \cdot n_+$ and $\nabla^- \psi \cdot n_-$ with respect to $\group^+$. 

The integration by parts formula is obtained by applying seperately on $\node^+_\theta$ and $\node^-_\theta$ the classical Green's identity. 

Finally from the symmetry properties \eqref{eq:symprops},  $\nabla^- \psi \cdot n_- =  - \nabla^+ \psi \cdot n_+$ on $\partial \node_\theta$ if and only if $\nabla^+ \psi \cdot n_+ \circ S = \nabla^+ \psi \cdot n_+$ or $\nabla^- \psi \cdot n_- \circ S = \nabla^- \psi \cdot n_-$ for any $S \in \group$ with $\det (S) = -1$.
\end{proof}
One can now apply these remarks to the fixed node groundstate $\psi^{\rm FN}_{\theta}$ solution of \eqref{eq:fixenodegroundstate}. 
\begin{Lemma}\label{lem:eigenequiv}
  Let $\psi^{\rm FN}_{\theta}$ be the solution of \eqref{eq:fixenodegroundstate} with a smooth boundary $\partial \node_\theta$. Then the following assertions are equivalent:
  \begin{enumerate}
  \item $\psi^{\rm FN}_{\theta}$ is an eigenfunction of $H$ in $\L^2(\R^d)$.
  \item $\nabla^{\rm sk} \psi^{\rm FN}_{\theta} \cdot n$, as defined by \eqref{eq:nskew}, vanishes on $\partial \node_\theta$.
  \item $\nabla^+ \psi^{\rm FN}_{\theta} \cdot n_+$ or $\nabla^- \psi^{\rm FN}_{\theta} \cdot n_-$ are symmetric under the action of $\group$ on $\partial \node_\theta$.
  \end{enumerate}
\end{Lemma}
\begin{proof}
The third point and the second point are equivalent by Lemma~\ref{lem:symbreak}. Now, if $\ph \in C^{\infty}_c(\R^d)$, integration by parts \eqref{eq:nodalipp} yields:
$$
\int_{\R^d} (H-E^{\rm FN}_{\theta})(\ph) \psi^{\rm FN}_{\theta} =\int_{\partial \node_\theta} \ph \nabla^{\rm sk} \psi^{\rm FN}_{\theta} \cdot n d \sigma,
$$
so that $\psi^{\rm FN}_{\theta} \in \L^2(\R^d)$ is an eigenfunction if and only if $\nabla^{\rm sk} \psi^{\rm FN}_{\theta} \cdot n =0$ on $\partial \node_\theta$.
\end{proof}
In the present context, the formula of the shape derivative of the Dirichlet groundstate can be written with symmetrized normal derivatives.  We denote the shape derivative $r_\theta^+$ (resp. $r_\theta^-$) of the nodes as defined in \eqref{eq:shapederiv},  with respect to the normal $n_+$ (resp. $n_-$), so that
$$
r_\theta^+ = - r_\theta^-.
$$
\begin{Proposition} The shape derivatives of the fixed node groundstate $\psi^{\rm FN}_{\theta}$ with respect to the nodal parameter $\theta$ reads
  \begin{equation}\label{eq:varfixednode}
\nabla_\theta E_\theta^{\rm FN} =    -\frac{1}{ \dps \int_{\node_\theta^+}(\psi^*_\theta) ^2}
  \int_{\partial \Omega_\theta}  \pare{ \nabla^{\rm sk} \psi^{\rm FN}_{\theta} \cdot n}\pare{ \nabla^{\rm sy} \psi^{\rm FN}_{\theta} \cdot n_+    }  r_{\theta}^+  \d \sigma .
\end{equation}
Moreover, $\nabla_\theta E_\theta^{\rm FN}=0$ for any parametrization $\theta \mapsto \partial \Omega_\theta$ of the nodes if and only if the fixed node groundstate $\psi^{\rm FN}_\theta$ is an eigenstate of $H$ in $\R^d$.
\end{Proposition}

\begin{proof}
The shape derivative formula \eqref{eq:vargroundstate} can be decomposed as the sum of the part due to $\node_\theta^+$ and to $\node_\theta^-$:
$$
 \nabla_\theta E_\theta^{\rm FN} = - \frac{1}{\dps 2\int_{\node_\theta^+ \cup \node_\theta^- }(\psi^*_\theta) ^2 }  \pare{  \int_{\partial \node_\theta}  \pare{ \abs{ \nabla^+ \psi^{\rm FN}_\theta  }^2 r_\theta^+  + \abs{ \nabla^- \psi^{\rm FN}_\theta  }^2  r_\theta^- }  \, d\sigma }.
$$
Since $\psi^*_\theta=0$ on $\partial \node_\theta$, then $\abs{\nabla^+ \psi^{\rm FN}_\theta} = \abs{\nabla^+ \psi^{\rm FN}_\theta \cdot n_+ }$ and symmetrically $\abs{\nabla^- \psi^{\rm FN}_\theta} = \abs{\nabla^- \psi^{\rm FN}_\theta \cdot n_- }$. We get:
\begin{align*}
  4 r_\theta^+ \pare{ \nabla^{\rm sk} \psi^{\rm FN}_{\theta} \cdot n}\pare{ \nabla^{\rm sy} \psi^{\rm FN}_{\theta} \cdot n_+    } & = 
 r_\theta^+ \pare{ \abs{ \nabla^+ \psi^{\rm FN}_\theta  }^2   - \abs{ \nabla^- \psi^{\rm FN}_\theta}^2 } \\
& =  r_\theta^+ \abs{ \nabla^+ \psi^{\rm FN}_\theta  }^2   + r_\theta^- \abs{ \nabla^- \psi^{\rm FN}_\theta}^2 ,
\end{align*}
where in the last line we have used $r_\theta^+ =-r_\theta^-$. Then \eqref{eq:varfixednode} follows. As a consequence, $\nabla_\theta E_\theta^{\rm FN} =0$ if and only if $\nabla^+ \psi^{\rm FN}_\theta \cdot n_+ = \nabla^- \psi^{\rm FN}_\theta \cdot n_- $, which is equivalent by Lemma~\ref{lem:eigenequiv} to the fact that $\psi^{\rm FN}_\theta$ is an eigenfunction.
\end{proof}
We can now state the main result of this paper, which consists in a characterization of the nodes of eigenstates through the symmetry of a random stopped process, and suggests a method to evaluate the shape derivative of the fixed node energy $ \nabla_\theta E_\theta^{\rm FN}$.
\begin{Theorem}\label{the:varFNproba}
Consider a Brownian motion $t \mapsto W_t^+$ in $\node_\theta^+$, and $\tau^+$ the hitting time of $\partial \node_\theta$.  Consider the measure $\mu^{\rm FN}_{\theta,\lambda}$ on $\partial \node_\theta$ defined for any $\lambda < E_\theta^{\rm FN}$ by \eqref{eq:munode}.
Then the fixed node groundstate $\psi^{\rm FN}_\theta$ is an eigenfunction of $H$ in $\R^d$, if and only if $\mu^{\rm FN}_{\theta,\lambda}$ is invariant under the symmetry group $\group$. Moreover, the shape derivative of the fixed node energy $E_\theta^{\rm FN}$ is given by \eqref{eq:partialEFN}.
\end{Theorem}
\begin{proof}
  The proof consists in a direct application of Lemma~\ref{lem:probaground} and Proposition~\ref{pro:hitdist}. Indeed, \eqref{eq:hitdist} yields the identity of measures on $\partial \node_\theta$:
$$
d \mu_{\theta, \lambda}^{\rm FN}  = -\frac{  \nabla^+ \psi^{\rm FN}_\theta \cdot n_+ \, d\sigma}{\dps 2 (E^{\rm FN}_\theta - \lambda   )\int_{\node_\theta^+} \pare{ \psi^{\rm FN}_\theta}^2 }
;
$$
and Lemma~\ref{lem:eigenequiv} enables to characterize eigenfunctions of $H$ from the symmetry of $\eta_\theta^{\rm FN, \lambda}$. Next, since $r^+_\theta$ on $\partial \node_\theta$ is skew-symmetric, and the integral on $\partial \node_\theta$ of integrable skew-symmetric functions vanishes, it yields:
$$
\int _{\partial \node_\theta} r^+_\theta   \nabla^{\rm sy} \psi_\theta^{\rm FN} \cdot n_+ \nabla^{+} \psi^{\rm FN}_\theta \cdot n_+ \, d\sigma =  \int_{\partial \node_\theta}  r^+_\theta  \nabla^{\rm sy} \psi^{\rm FN}_\theta \cdot n_+  \nabla^{\rm sk} \psi^{\rm FN}_\theta \cdot n \, d\sigma .
$$
Then \eqref{eq:varfixednode} with \eqref{eq:proba_psistar2} yield the result \eqref{eq:partialEFN}.
\end{proof}

\section{Comments on Monte-Carlo methods}\label{sec:montecarlo}
Theorem~\ref{the:varFNproba} suggests a Monte-Carlo general strategy to approximate the fixed node energy variation \eqref{eq:partialEFN} by using the probabilistic interpretation \eqref{eq:munode}-\eqref{eq:etanode}. As already commented in the introduction, computing \eqref{eq:munode} with Monte-Carlo methods is a well known topic, and
computing \eqref{eq:etanode} can be carried out using independent stopped processes. In fact, the most straightforward limitation of the method consists in approximating on the nodes $\partial \node_\theta$ the symmetrized gradient 
$$
\nabla^{\rm sy} \psi^{\rm FN}_\theta \sim \nabla \psi^n_\theta
$$
by a sequence $\psi_\theta^n \xrightarrow[]{n \to + \infty} \psi^{\rm FN}_\theta$ of smooth approximating functions in $\R^d$, usually converging in the sense of the energy norm. Indeed, there is no reason {\it a priori} that the gradient $\nabla \psi^n_\theta$ converges in a pointwise sense towards $\nabla^{\rm sy} \psi^{\rm FN}_\theta$, which is necessary when approximating the measure $\eta_\theta^{\rm FN, \lambda}$ using a Monte-Carlo sample. However, in the context of optimization, it is fundamental to remark that the stationary nodal surfaces solution of
$
\nabla_\theta E_\theta^{\rm FN} = 0
$
for any shape derivative are not modified when approximating $\nabla^{\rm sy} \psi^{\rm FN}_\theta \cdot n_+$ by any symmetric field $\nabla \psi^n_\theta \cdot n^+$. This is the meaning of the following proposition. 
\begin{Proposition} Consider any smooth skew-symmetric function $\psi_\theta^{\rm I} $  in $\R^d$ with zeros given by $\partial \node_\theta$. Then the approximation $\widehat{\nabla_\theta E_\theta^{\rm FN}}$ given by \eqref{eq:partialEFNapproxdef} yields \eqref{eq:partialEFNapprox}-\eqref{eq:partialEFNapproxbis}. Moreover the fixed point equation
$$
 \widehat{\nabla_\theta E_\theta^{\rm FN}} = 0
$$
holds for any shape variation $r^+_\theta$ of the nodes $\partial \node_\theta$ if and only if $\psi_\theta^{\rm FN} $ is an eigenfunction of the Hamiltonian $H$ in $\R^d$.
\end{Proposition}
\begin{proof}
Since $\partial \node_\theta$ is defined implicitely by $\psi^{\rm I}_\theta = 0$, deriving $\psi^{\rm I}_{\theta+h}(x+r_{\theta + h}^+ n_+) =0$ with respect to $h$ yields on $\partial \node_\theta$:
$$
\nabla_\theta \psi^{\rm I}_{\theta} + r_{\theta}^+  \nabla \psi^{\rm I}_{\theta} \cdot n_+ = \nabla_\theta \psi^{\rm I}_{\theta} -  r_{\theta}^+  \abs{\nabla \psi^{\rm I}_{\theta} } = 0,
$$
where in the last line we have used the fact that $\nabla \psi^{\rm I}_{\theta} \cdot n_+=-\abs{\nabla \psi^{\rm I}_{\theta} }$ on $\partial \node_\theta$ since $\node_\theta^+= \set{ x\in \R^d \, \vert \, \psi^{\rm I}_{\theta}(x) > 0}$.
This gives \eqref{eq:partialEFNapprox}. Next \eqref{eq:hitdist} yields:
$$
\arr{rl}{
\int_{\partial \node_\theta^+}  \nabla_\theta \psi^{\rm I}_{\theta} \d\mu^{\rm FN}_{\theta,\lambda} 
= &\dps -\frac{1}{\dps 2(E^{\rm FN}_\theta - \lambda   ) \int_{\node_\theta^+} \psi^{\rm FN}_\theta} 
   \int_{\partial \node_\theta}  \nabla_\theta \psi^{\rm I}_{\theta} \nabla^+ \psi^{\rm FN}_\theta \cdot n_+ \d\sigma ,
}
$$
so that integration by parts in $\node_\theta^+$ gives \eqref{eq:partialEFNapproxbis}.

 Finally remark that $\nabla \psi_\theta^{\rm I} \cdot n_-$ is symmetric and that the shape derivative $r^+_\theta$ is spanning all skew-symmetric field, so that the fixed point equation is equivalent to the fact that $\mu_\theta^{\rm FN}$ is symmetric. Then Lemma~\ref{lem:eigenequiv} enables to conclude.
\end{proof}

Eventually, the following algorithm (without details) is suggested to estimate $\nabla_\theta E_\theta^{\rm FN}$. It might be referred to as Nodal Monte-Calo (NMC). 
\begin{Algorithm}[NMC]
  Consider a parametrization of wave functions of the form \eqref{eq:trial}.
Then the following steps are suggested:
\begin{enumerate}
\item Optimize the symmetric ''Jastrow'' factor $ J_\alpha $ according to the variational problem with fixed node \eqref{eq:fixenodegroundstate}, and denote $ \psi_{ \theta}^{\rm I}$ the obtained trial function.
\item Generate a sample according to the probability distribution $\eta^{\rm FN}_{\theta}$ in $\node_\theta^+$ defined in \eqref{eq:munode}. Use for instance a long time trajectory of a drifted stochastic process of the type \eqref{eq:drifteds} with Feynman-Kac weights as in formula \eqref{eq:proba_psistar2}.
\item Use the latter sample as initial conditions; and sample independent Brownian motions. Stop them when they hit the nodes $\partial \node_\theta$. Then compute the measure $\mu^{\rm FN}_{\theta,\lambda}$ according to \eqref{eq:munode}.
\item Estimate the variations of the fixed node energy $\nabla_\theta E_\theta^{\rm FN}$ with \eqref{eq:partialEFNapprox}.
\end{enumerate}
\end{Algorithm}
As a conclusion, we mention two possible strategies of variance reduction that will be necessary for efficient computations. Their development is left for future work.
  \begin{itemize}
  \item First, a drifted diffusion instead of plain Brownian motion may be used in \eqref{eq:munode} to reduce the variance caused by the exponential weights. However, the classical stochastic differential equation \eqref{eq:drifteds} where $\psi^{\rm I}$ vanishes at the nodes cannot be used. Indeed, the repulsive drift at the nodes prevent the process from hitting the latter, and \eqref{eq:munode} no longer holds. Instead of a skew-symmetric guiding function $\psi^{\rm I}$, we propose to use as a drift a symmetric function $\psi_B$, for instance an approximation of the \emph{Bosonic} groundstate $\psi^\ast_{\rm B} > 0$ of $H$ in $\R^d$ solution of the eigenvalue problem
$$
H(\psi^\ast_{\rm B}) = E^{\ast}_{\rm B} \psi^\ast_{\rm B}.
$$ 
This may remove the variance of the exponential Feynman-Kac weights, while letting the walkers hit the nodal surface. Such variance reductions are referred to as \emph{importance sampling} methods. Note that in known algorithms ({\it e.g.} DMC), such guided walkers will very quickly hit the nodal surface, causing high variance branching. However, in the proposed method, the algorithm is stopped when walkers have hit the nodal surface once, and no branching is performed, avoiding such kind of variance instability.
\item Second, a key point would be to develop a coupling (in the probabilistic sense) between the random process used to compute $\mu^{\rm FN}_{\theta}$ in \eqref{eq:munode} and another random process stopped on the nodes $\partial \node_\theta$, with a distribution denoted $\widetilde{\mu}^{\rm FN}_{\theta}$. The goal is to design $\widetilde{\mu}^{\rm FN}_{\theta}$ so that the following two features hold: first computing the average \eqref{eq:partialEFNapprox} with $\widetilde{\mu}^{\rm FN}_{\theta}$ always yields $0$; second the coupling is perfect in the special case when $\mu^{\rm FN}_{\theta}$ is symmetric (that is to say the Monte-Carlo method computing $\mu^{\rm FN}_{\theta}$ and $\widetilde{\mu}^{\rm FN}_{\theta}$ has to be the same with the same random numbers). In the physics terminology, this would yield a \emph{zero variance estimator}, which we prefer to call \emph{asymptotic variance reduction}, in the sense that thanks to the variance reduction, the variance of the estimator scales with the quantity to be computed when the latter becomes small.
  \end{itemize}
\section*{Acknowledgements} The author would like to thank E. Cancès, B. Jourdain and T. Lelièvre for pointing out the problem and fruitfull discussions, as well as the referees for very carefull reading and useful suggestions.
\bibliographystyle{plain}
\bibliography{nodes}

\end{document}